\documentclass[reqno]{amsart}
\usepackage[margin = 1.3in]{geometry}
\usepackage{amsmath, amssymb, amsthm, fancyhdr, verbatim, graphicx}
\usepackage{enumerate}
\usepackage{caption}
\usepackage[all]{xy}
\usepackage[dvipsnames]{xcolor}
\usepackage{mathrsfs}
\usepackage{tikz-cd}
\usepackage{framed}
\usepackage[hidelinks, pagebackref]{hyperref}
\hypersetup{colorlinks=true, linkcolor=Violet,citecolor=Emerald}
\usepackage[titletoc]{appendix}
\usepackage{bbm}
\usepackage{lipsum}
\usepackage{adjustbox}
\usepackage{mathfyz}
\usepackage{mathzyw}
\usepackage{envi}
\usepackage{stmaryrd}
\usepackage{leftindex}
\usepackage[english]{babel}
\numberwithin{equation}{section}

\makeatletter
\def\th@remark{%
  \thm@headfont{\bfseries}%
  \normalfont
  \thm@preskip \thm@preskip 
  \thm@postskip\thm@preskip
}
\def\imod#1{\allowbreak\mkern5mu({\operator@font mod}\,\,#1)}
\makeatother

\title[Characteristic $2$ Lie Type Divisibility Graph]{Components of the Divisibility Graph of Finite Groups of Lie Type in Defining Characteristic $2$}

\author{Liam May}
\address{Department of Mathematics, Massachusetts Institute of Technology, 77 Massachusetts Avenue, Cambridge, MA 02139, USA}
\email{liammay2@mit.edu}

\author{Yong Yang}
\address{Department of Mathematics, MCS 470, 601 University Drive, San Marcos, TX 78666-4684}
\email{yang@txstate.edu}

\begin{document}

\begin{abstract}
We determine the connected components of the divisibility graph of several families of finite groups of Lie type in defining characteristic $2$, extending the result of \cite{AbdolghafourianIranmaneshNiemeyer2017LieType}, who treated odd defining characteristic. We show that there is a distinguished component containing all nonidentity unipotent class sizes and that every other component is either an isolated vertex or an explicitly described two-vertex component. We also determine the isolated vertices outside the distinguished component. Additionally, we prove in Appendix~\ref{app:frobenius} that the divisibility graph of a Frobenius group has exactly two components.
\end{abstract}

\maketitle
\tableofcontents

\section{Introduction}

Given a finite set $X$ of positive integers, define the divisibility graph $D(X)$ as follows. The vertices are elements of $X$ not equal to $1$, and there is an edge $a\sim b$ if $a\mid b$ or $b\mid a$. We will study $D(X)$ in the case where $X = \cs(G)$ for a finite group $G$, where $\cs(G)$ denotes the set of conjugacy class sizes of $G$, and we let $\cD(G)$ denote this graph.

It is a natural question to ask what the number of connected components, $c(\cD(G))$, of $\cD(G)$ can be, which is stated as Question $7$ in \cite{CaminaCamina2011Survey}. In \cite{AbdolghafourianIranmanesh2015SymAlt}, it is shown that $c(\cD(S_n))\le 2$, and $c(\cD(A_n))\le 3$, and in the disconnected case all but one component is an isolated vertex. Similarly, $c(\cD(G))$ is studied for simple Zassenhaus groups and sporadic simple groups in \cite{AbdolghafourianIranmanesh2014Zassenhaus}, the groups $\PSL(2,q)$ and $\operatorname{Sz}(q)$ in \cite{AbdolghafourianIranmanesh2016CertainSimple}, and more generally, finite groups of Lie type in ``good'' defining characteristic in \cite{AbdolghafourianIranmaneshNiemeyer2017LieType}. The authors of \cite{khoshnevis_mostaghim_2022_fgroups} study $F$-groups. In this paper, we consider finite groups of Lie type in defining characteristic $2$; these were not considered in \cite{AbdolghafourianIranmaneshNiemeyer2017LieType}. In Appendix~\ref{app:frobenius}, we also present a short proof of the fact that $c(\cD(G)) = 2$ when $G$ is a Frobenius group, addressing another apparent gap in the literature.

The authors of \cite{AbdolghafourianIranmaneshNiemeyer2017LieType} study the divisibility graph of several families of finite groups of Lie type, however their discussion entirely excludes those of even characteristic. With the exception of $\PSL_2(q)$ and $\operatorname{Sz}(q)$, whose divisibility graphs are given in \cite{AbdolghafourianIranmanesh2016CertainSimple}, nothing is published on the divisibility graphs of Lie-type groups in even characteristic.

\medskip

As observed in \cite{AbdolghafourianIranmaneshNiemeyer2017LieType}, for $x,y\in G$ with $|x^G|\ne |y^G|$, there is an edge between the corresponding vertices of $D(\cs(G))$ if and only if $|C_G(x)|\mid |C_G(y)|$ or $|C_G(y)|~\big\vert~|C_G(x)|$. So we can replace the set of nontrivial conjugacy class sizes of $G$ with the set $\mathcal C(G)$ of centralizer sizes of noncentral elements of $G$. We write $x\sim y$ if $|x^G|$ and $|y^G|$ lie in the same component of $D(\cs(G))$, and we may also view this as an equivalence relation on the set $\cC(G)$.

The following is Lemma 4 of \cite{AbdolghafourianIranmaneshNiemeyer2017LieType}.

\begin{lemma}\label{lemma:4-AIN}
    Let $G$ be a finite group of Lie type and $x,y \in G$ noncentral elements.
    \begin{itemize}
        \item If $\gcd(|x|, |y|) = 1$ and $xy = yx$, then $C_G(xy) = C_G(x) \cap C_G(y)$, thus $x\sim xy\sim y$.
        \item $x\sim x^m$ for any $m > 1$ such that $x^m \notin Z(G)$.
        \item If $\gcd(|xZ(G)|, |yZ(G)|) = 1$ and $xyZ(G) = yxZ(G)$, then $x\sim y$.
    \end{itemize}
\end{lemma}

Due to the Jordan decomposition $g = su = us$ of a noncentral element $g\in G$, and Lemma~\ref{lemma:4-AIN}, either $s\notin Z(G)$, in which case $g\sim s$, or $u$ is nontrivial, in which case $g\sim u$. Hence, much of the focus of this paper, following the structure of \cite{AbdolghafourianIranmaneshNiemeyer2017LieType}, is on unipotent and semisimple elements of $G$.

\subsection{Outline of Paper}
In Section~\ref{sec:background}, we recall the basic definitions of linear algebraic groups and finite groups of Lie type and establish the scope of the paper. We prove that semisimple elements of $G$ have odd order, that unipotent elements are exactly those of $2$-power order, and that the center has odd order. In Section~\ref{sec:unip}, we show that all nonidentity unipotent elements have class sizes lying in a single component, called the \emph{unipotent component} $\cU$. In Section~\ref{sec:4}, we show that noncentral nonregular semisimple elements and noncentral semisimple elements with even centralizer order lie in the unipotent component, and establish a further technical lemma. Sections~\ref{sec:5} and~\ref{sec:6} handle the remaining cases. The main theorems are proved in Section~\ref{sec:7}.

\section{Background and Scope} \label{sec:background}

We now recall the basic terminology of linear algebraic groups and finite groups of Lie type, taken from \cite{MalleTesterman2011}, and fix the notation used throughout. In this paper, $k$ denotes $\overline \F_2$, $q$ is a fixed positive power of $2$, and we take all varieties to be over $k$. We do not take the convention that a variety must be irreducible. We say that an affine variety $\sG$ over $k$ is a \emph{linear algebraic group} if $\sG$ is equipped with a group law such that multiplication $\sG\times \sG \to \sG$ and inversion $\sG \to \sG$ are morphisms of varieties. The linear algebraic group $\sG$ is called \emph{connected} if $\sG$ is connected as a variety, which is true if and only if $\sG$ is irreducible as a variety. The component of $\sG$ containing the identity is denoted $\sG^\circ$. 

The general linear groups, $\GL(V)$, are indeed linear algebraic groups, and every linear algebraic group $\sG$ can be viewed as a closed subgroup of a general linear group, that is, there exists a faithful rational representation $\rho:\sG \hookrightarrow \GL(V)$ for some finite dimensional vector space $V$ over $k$. An element $g\in \GL(V)$ is said to be \emph{unipotent} if $g - 1$ is nilpotent, and \emph{semisimple} if $g$ is diagonalizable. Moreover, for all $g\in \GL(V)$, there exists a unique \emph{Jordan decomposition}, $g = g_u g_s = g_s g_u$, where $g_u$ is unipotent and $g_s$ is semisimple, the two factors commute, and are called the \emph{unipotent} and \emph{semisimple} components, respectively. For an element $x$ of $\sG$, there exists a unique decomposition $x = x_u x_s = x_s x_u$, where $\rho(x_u)$ is unipotent and $\rho(x_s)$ is semisimple, for a faithful rational representation $\rho$, which is also called the \emph{Jordan decomposition}, and moreover is independent of the choice of representation. An element $x$ is said to be \emph{unipotent} if $x = x_u$, and similarly \emph{semisimple} if $x = x_s$. A group is said to be \emph{unipotent} if all elements are unipotent.

The \emph{radical} of $\sG$, denoted $R(\sG)$, is the largest closed, connected, normal, solvable subgroup of $\sG$. The \emph{unipotent radical} of $\sG$, denoted $R_u(\sG)$, is the largest closed, connected, normal, unipotent subgroup of $\sG$. A connected linear algebraic group $\sG$ is called \emph{reductive} if $R_u(\sG)=1$, and it is called \emph{semisimple} if $R(\sG)=1$. Since $R_u(\sG)\le R(\sG)$, every semisimple group is reductive.

The multiplicative group of $k$, denoted $\sG_m$, can easily be seen to be a linear algebraic group. A linear algebraic group $\sT$ is called a \emph{torus} if it is isomorphic to a product of copies of $\sG_m$. In this paper, assume that $\sG$ is connected and semisimple, and that $\sT\le \sG$ is a maximal torus, unless specified otherwise. A \emph{Borel subgroup} of $\sG$ is a maximal closed, connected, solvable subgroup of $\sG$. Often, we will consider $\sT\le \sB\le \sG$, where $\sT$ is a maximal torus, $\sB$ is a Borel subgroup, and both are $F$-stable for some Steinberg endomorphism $F$ on $\sG$. A \emph{character} of $\sT$ is a morphism $\chi:\sT\to \sG_m$, and the set of characters is denoted $X=X(\sT)$. The \emph{Lie algebra} of $\sG$ will be denoted $\frg=\Lie(\sG)$, which we identify with the tangent space $T_1\sG$. Given $x\in\sG$, the map $y\mapsto xyx^{-1}$ gives rise to a map $\Ad x:\frg\to\frg$, which defines the \emph{adjoint representation} $\Ad:\sG\to\GL(\frg)$. The image $\Ad(\sT)$ consists of commuting semisimple elements and therefore can be simultaneously diagonalized. Thus we write $\frg=\bigoplus_{\chi\in X(\sT)}\frg_\chi$, where $\frg_\chi=\{v\in\frg\mid \Ad(t)(v)=\chi(t)v\text{ for all }t\in\sT\}$. The set of nontrivial characters $\chi$ such that $\frg_\chi\ne0$ is denoted $\Phi=\Phi(\sG,\sT)$ and is called the set of \emph{roots} of $\sG$ with respect to $\sT$. The \emph{Weyl group} with respect to $\sT$ is $W=N_{\sG}(\sT)/C_{\sG}(\sT)$, and it has a well-defined action on the set of roots. 

We also consider the \emph{cocharacters} of $\sT$, which are morphisms $\gamma:\sG_m\to\sT$, and denote the cocharacter lattice by $Y=Y(\sT)=\Hom(\sG_m,\sT)$. A \emph{Dynkin diagram} is a graph that encodes the geometric data of the roots of $\sG$. Each root $\alpha \in \Phi(\sT)$ corresponds to a \emph{reflection} $s_\alpha$ of the vector space $X(\sT) \otimes_\Z \R$, which gives rise to a \emph{coroot} $\alpha^\vee$. The set of coroots is denoted $\Phi^\vee = \Phi(\sT)^\vee$, and the data $(X, \Phi, Y, \Phi^\vee)$ is called a root datum. A theorem of Chevalley states that two semisimple linear algebraic groups are isomorphic if and only if they have isomorphic root data. The possible Dynkin diagrams which may occur are labeled $A_n, B_n, C_n, D_n, G_2, F_4, E_6, E_7, E_8$. Groups corresponding to the first four families are said to be of \emph{classical type}, and the rest of \emph{exceptional type}. 

We can study which root data arise with a given root system $\Phi$. Let $\Omega=\Hom(\Z\Phi^\vee,\Z)$, and let $\Lambda(\Phi)=\Omega/\Z\Phi$, the fundamental group of the root system. If $(X,\Phi,Y,\Phi^\vee)$ is a root datum, then $\Z\Phi\le X$ and $\Z\Phi^\vee\le Y$ are finite-index subgroups. The latter inclusion induces an inclusion $X\cong\Hom(Y,\Z)\hookrightarrow\Omega$. Thus the possible character lattices satisfy $\Z\Phi\le X\le\Omega$ and correspond to subgroups of $\Lambda(\Phi)$. The algebraic fundamental group of $\sG$ is $\Lambda(\sG)=\Omega/X$.

Now we pass from linear algebraic groups to finite groups of Lie type. Given an affine variety $\sX$ over $k$, the Frobenius map on $k$ defines a map $\sX\to \sX$, also called the Frobenius map. Given $\sG$, a connected semisimple linear algebraic group, we say that a map $F:\sG \to \sG$ is a \emph{Steinberg endomorphism} if $F^m$ is a Frobenius map with respect to some $\F_q$ structure, where $q = 2^n$, for some $m\ge 1$. Then a \emph{finite group of Lie type} is formed as $G = \sG^F$, the fixed points of $F$ in $\sG$. The \emph{type} of a finite group $G=\sG^F$ of Lie type is written ${}^\delta R(q)$, where $R$ is the Lie type of $\sG$, $q$ is the prime power over which $G$ is defined, and $\delta$ is the order of the automorphism of the Dynkin diagram of $\sG$ induced by $F$. The main theorems consider exactly the set of groups considered in \cite{AbdolghafourianIranmaneshNiemeyer2017LieType}, with $q$ a power of $2$.

We now state the groups that the main theorems consider. The group $G$ denotes a finite group of the form $G = \sG^F$, where $\sG$ is a connected semisimple linear algebraic group over $k$, and $F:\sG \to \sG$ is a Steinberg endomorphism. We consider groups of classical types $A_\ell$ and ${}^2 A_\ell$ for $\ell \ge 3$, $C_\ell$ for $\ell \ge 3$, and $D_\ell$ and ${}^2D_\ell$ for $\ell\ge 3$. We also consider the exceptional types $G_2$, $F_4$, $E_6$, ${}^2 E_6$, $E_7$, $E_8$, ${}^3 D_4$, ${}^2 B_2$, and ${}^2 F_4$. All groups are defined in characteristic $2$. Note that for $q$ even, there is a standard identification $\Omega_{2\ell + 1} (q) \cong \Sp_{2\ell}(q)$, hence considering the families $B_\ell$ separately from $C_\ell$ would be superfluous. There are the following further low-rank coincidences: $A_1  =C_1$, $B_2 = C_2$, $D_3 = A_3$, and ${}^2 D_3 = {}^2 A_3$.

In Section~\ref{sec:7}, we prove that the graph $\cD(G)$ has one distinguished connected component containing all nontrivial unipotent class sizes, and that every other component is either an isolated vertex or one of the explicit two-vertex components arising from irreducible tori. We also classify the isolated vertices outside the distinguished component.

The following two lemmas give useful tools relating important element properties with element order. Lemma~\ref{lemma:2-1} allows us to identify unipotent elements exactly as those with order a power of $2$, and notes that semisimple elements have odd order. Unipotent elements play an important role in the proof of the main theorem, as the distinguished component of $\cD(G)$ will be built around unipotent elements. Moreover, Lemma~\ref{lemma:2-2} shows that the noncentral unipotent elements are the same as the nontrivial unipotent elements.

\begin{lemma}\label{lemma:2-1}
    Let $x\in \sG$ have finite order. If $x$ is semisimple, then its order is odd. Furthermore, $x$ is unipotent if and only if its order is a power of $2$.
\end{lemma}

\begin{proof}
    By \cite[\S 5.2]{MalleTesterman2011}, there exists a faithful rational representation $\rho : \sG \to \GL(V)$ of $\sG$. Let $x\in \sG$ be semisimple and suppose its order $|x|$ is even. Then $x' := x^{|x| / 2}$ has order $2$ and is semisimple. The image $\rho(x')$ satisfies $\rho(x')^2 - 1 = (\rho(x') - 1)^2 = 0$, hence has all eigenvalues $1$. Since $\rho(x')$ is semisimple, $\rho(x') = 1$, contradicting $\rho$ being faithful. Therefore $|x|$ is odd.

    Suppose $x\in \sG$ is unipotent. Then $\rho(x) = I+N$ with $N$ nilpotent. For sufficiently large $r$, $N^{2^r} = 0$, thus $(I+N)^{2^r} = I + N^{2^r} = I$, so the order of $I+N$ divides $2^r$. Therefore $|x| \mid 2^r$. Conversely, suppose $x$, and therefore $\rho(x)$, has order $2^r$. Then $(I + \rho(x) - I)^{2^r} = I + (\rho(x) - I)^{2^r} = I$, so $x$ is unipotent, as $\rho(x) - I$ is nilpotent. 
\end{proof}

\begin{lemma}\label{lemma:2-2}
    Let $Z = Z(G)$ denote the center of $G$. Then $Z$ has odd order, and every nonidentity unipotent element of $G$ is noncentral. 
\end{lemma}

\begin{proof}
    Since $\sG$ is semisimple, $Z(\sG)$ is finite. Every maximal torus $\sT$ contains $Z(\sG)$, so $Z(\sG)$ consists only of semisimple elements, which by Lemma~\ref{lemma:2-1} have odd order. This demonstrates that $Z(\sG)$ has odd order.

    Since $Z(\sG)^F = Z(\sG^F)$ \cite[24.13]{MalleTesterman2011}, the subgroup $Z(G) = Z(\sG)^F$ of $Z(\sG)$ also has odd order. But unipotent elements have order a power of $2$, so no nontrivial unipotent element lies in $Z(G)$.
\end{proof}

\section{The Unipotent Component}\label{sec:unip}

In this section we will establish the existence of a distinguished \emph{unipotent component}, denoted $\cU$, of the divisibility graph $\cD(G)$, which contains all noncentral unipotent class sizes. 

\begin{prop}[Unipotent Component]\label{prop:unip-component}
    There is a nonidentity unipotent element $u\in G$ such that $|C_{G}(u)|\mid|C_{G}(v)|$ for every nonidentity unipotent $v\in G$. Consequently, the class sizes of all nonidentity unipotent elements of $G$ lie in a single connected component of $\cD(G)$, which we denote $\cU$.
\end{prop}

In the classical cases, Lemma~\ref{lemma:unip-classical} shows that a single Jordan block serves as a unipotent element whose centralizer size divides those of all other noncentral unipotent elements. In the exceptional cases, inspection of the data presented in \cite{liebeckSeitz2012}, \cite{spaltenstein1982}, and \cite{suzuki1962} allows us to conclude that a regular unipotent element provides the desired centralizer order; this is established in Lemma~\ref{lemma:exc-unip-reg-divis-1}. The proposition follows from these two lemmas.

\begin{lemma}\label{lemma:unip-classical}
     Suppose that $G$ belongs to one of the classical families under consideration. Then there exists a nonidentity unipotent element $u\in G$ such that $|C_G(u)|\mid|C_G(v)|$ for all nonidentity unipotent elements $v\in G$. 
\end{lemma}

\begin{proof}
    First consider the family $A_\ell$. We will show that without loss of generality, we may assume that such a group $G = \sG^F$ satisfies $\SL_n(q) \le G \le \GL_n(q)$. By \cite[Prop 9.15]{MalleTesterman2011}, there are central isogenies\[
        \SL_n(k)\xrightarrow{\alpha}\sG\xrightarrow{\beta}\PGL_n(k).
    \]
    Let $\overline G = \beta(G) \le \PGL_n(q)$, and $K = \ker(\beta|_G)$. Then $\PSL_n(q) \le \overline G \le \PGL_n(q)$. Let $\rho: \GL_n(q) \to \PGL_n(q)$ be the natural central isogeny, and set $\widetilde G = \rho^{-1}(\overline G)$. Then $\SL_n(q) \le \widetilde G \le \GL_n(q)$. For a unipotent element $\overline x \in \PGL_n(q)$, there is a unique representative $\hat x\in \SL_n(q)$. Indeed, two lifts differ by a scalar, and if both are unipotent, then that scalar is $1$. So fix $x\in G$ unipotent, $\overline x = \beta(x)\in \PGL_n(q)$, and $\hat x \in \SL_n(q)$ the lift.     We claim that $1\to K\to C_G(x)\to C_{\overline G}(\overline x)\to 1$ is exact. Only surjectivity requires proof. Let $\overline g\in C_{\overline G}(\overline x)$ and choose $g\in G$ such that $\beta(g)=\overline g$. Since $\overline g$ centralizes $\overline x$, there exists $z\in K$ such that $gxg^{-1}=zx$. The element $z$ is central and semisimple, while $x$ is unipotent. Thus the semisimple and unipotent components of $zx$ are $z$ and $x$, respectively. Since $gxg^{-1}$ is unipotent, uniqueness of the Jordan decomposition gives $z=1$. Hence $g\in C_G(x)$, proving surjectivity.

    The same argument shows that $1\to Z(\GL_n(q))\to C_{\widetilde G}(\hat x)\to C_{\overline G}(\overline x)\to 1$ is exact. Indeed, if $\overline g\in C_{\overline G}(\overline x)$ and $\hat g\in\widetilde G$ is a lift, then $\hat g\hat x\hat g^{-1}=c\hat x$ for some scalar $c$. Since the left-hand side is unipotent, uniqueness of the Jordan decomposition gives $c=1$.

    Consequently, $|C_G(x)|=|K||C_{\overline G}(\overline x)|$ and $|C_{\widetilde G}(\hat x)|=(q-1)|C_{\overline G}(\overline x)|$. Therefore divisibility of the relevant centralizer orders is equivalent in $G$, $\overline G$, and $\widetilde G$, proving the reduction.

    For the family $A_{n-1}$, consider a group $G$ with $\SL_n(q)\le G\le \GL_n(q)$. This is a larger family of groups than the scope of this paper treats. Let $v\in G$ be a nonidentity unipotent element. Thus $v\in \SL_n(q)$, and $v$ has a Jordan decomposition \[
        v\sim \bigoplus_{i = 1}^h J_{\lambda_i}^{\oplus l_i},
    \]
    where $\lambda_1 < \dots < \lambda_h$ and $\sum \lambda_i l_i = n$. From the short exact sequence \[
        1\to C_{\SL_n(q)}(v) \to C_G(v) \to \det(C_G(v)) \to 1,
    \]
    where $\det(C_G(v)) = \det(G) \cap \det(C_{\GL_n(q)}(v))$, we can write \[
        |C_G(v)| = |C_{\SL_n(q)}(v)|\cdot |\det(G) \cap \det(C_{\GL_n(q)}(v))|. \tag{1}\label{tag:1}
    \] 
    
    Now take $u\in G$ to the single Jordan block $u = J_n$. The aim is to show the desired centralizer divisibility, and we will show divisibility of both factors of the right hand side of $(\ref{tag:1})$. From \cite[Thm. 2.3.9]{deFranceschi2020}, \[
        |C_{\GL_n(q)}(u)| = q^{n-1}(q-1),
    \]
    and from \cite[Thm. 2.4.1]{deFranceschi2020}, \[
        |C_{\SL_n(q)}(u)| = \gcd(n, q-1)q^{n-1}.
    \]
    Similarly, the $q$-part of $|C_{\SL_n(q)}(v)|$ is $q^{\gamma + \sum {\tbinom{l_i}{2}}}$, where $\gamma = 2\sum_{i < j}\lambda_i l_il_j + \sum_i (\lambda_i - 1)l_i^2$. Since $\gamma + \sum {\tbinom{l_i}{2}} \ge n - 1$ divisibility of the $q$-part of the first factor of ${(\ref{tag:1})}$ follows. The $q'$-part of the centralizer order in $\GL_n(q)$ is given by \[
        |C_{\GL_n(q)}(v)|_{q'} = (q - 1)^h \prod_{i = 1}^h \prod_{k = 2}^{l_i} (q^k - 1),
    \]
    and thus  \[
        |C_{\SL_n(q)}(v)|_{q'} = d(q - 1)^{h-1} \prod_{i = 1}^h \prod_{k = 2}^{l_i} (q^k - 1),
    \]
    where $d = \gcd(\lambda_1,\dots, \lambda_h, q - 1)$. If $u = v$, divisibility is clear. Otherwise, $h\ge 2$ or $h=1$ and $l_1\ge2$. In either case it is clear that the $q'$-part is divisible by $q - 1$ and thus by $\gcd(n, q-1)$. This shows divisibility of the $q'$-part of the first factor of ${(\ref{tag:1})}$, so divisibility of this factor follows. Now we check divisibility of the second factor of $(\ref{tag:1})$. An element $(A_1,\dots,A_h)$ of $\prod_i\GL_{l_i}$ acts on the $J_{\lambda_i}^{\oplus l_i}$ summand of the Jordan decomposition with determinant $\det(A_i)^{\lambda_i}$. From the description of the $\GL_n(q)$-centralizer in \cite{deFranceschi2020}, noting that the unipotent components contribute determinant $1$, we get\[
        \det \left(C_{\GL_n(q)}(v)\right) = \prod_i(\F_q^\times)^{\lambda_i} = (\F_q^\times)^{g(v)},
    \] where $g(v) = \gcd(\lambda_1,\dots,\lambda_h)$ (note the above expression is not a direct product, but simply a product of groups). Similarly, \[
        \det \left(C_{\GL_n(q)}(u)\right) = (\F_q^\times)^n.
    \] Noticing that $g(v) \mid \sum \lambda_i l_i = n$, and hence \[
        \det (G) \cap (\F_q^\times)^n \le \det (G) \cap (\F_q^\times)^{g(v)},
    \]  we complete the proof in the $A_{n-1}$ case.

    Next we consider the family ${}^2 A_{n-1}$, considering groups $\SU_n(q)\le G\le \GU_n(q)$, where we use the same reduction used in the type $A_{n-1}$ case. The argument is nearly identical, so we omit some of the details. Let $H = \{a\in \F_{q^2}^\times \mid a^{q + 1} = 1\}$. We can again write \[
        |C_G(v)| = |C_{\SU_n(q)}(v)|\cdot|\det (G) \cap \det \left(C_{\GU_n(q)}(v)\right)|.
    \]
    Choosing $u$ and $v$ as before, with the same Jordan decomposition, from \cite[Thm. 5.2.1]{deFranceschi2020}, we get \[
        |C_{\SU_n(q)}(u)| = \gcd(n, q + 1) q^{n-1}
    \]
    and \[
        |C_{\SU_n(q)}(v)|_q = q^{2\gamma + \sum_i \tbinom{l_i}{2}},
    \]
    with $\gamma$ as defined in \cite{deFranceschi2020}, where $2\gamma + \sum {\tbinom{l_i}{2}} \ge n - 1$. For the $q'$-part, \[
        |C_{\SU_n(q)}(v)|_{q'} = d(q + 1)^{h-1}\prod_i \prod_{k = 2}^{l_i} (q^k - (-1)^k).
    \]
    If $h \ge 2$, there is a factor of $q + 1$, and if $h=1$ and $l_1\ge2$, there is a factor of $q^2 - 1 = (q + 1)(q - 1)$. Thus we have shown divisibility of the $\SU_n$ factor of the centralizer order. For the second factor, we find \[
        \det \left(C_{\GU_n(q)}(v)\right) = H^{g(v)},
    \]
    where $g(v) = \gcd(\lambda_1, \dots, \lambda_h)$, and \[
        \det \left(C_{\GU_n(q)}(u)\right) = H^n.
    \]
    Thus $\det (G)\cap H^n \le \det (G) \cap H^{g(v)}$ gives the necessary divisibility.

    The remaining cases of $C_\ell, D_\ell$, and ${}^2D_\ell$, are again identical in structure, using \cite[Thm. 5.2.2, Thm. 5.17, Thm. 5.2.1]{deFranceschi2020}. In type $C_\ell$, let $V$ be equipped with a nondegenerate alternating form $B$, and let $\mu: \GSp(V) \to \F_q^\times$ be defined by $B(gv, gw) = \mu(g)B(v,w)$, so $\ker\mu = \Sp(V)$. In types $D_\ell$ and ${}^2D_\ell$, let $V$ be equipped with the corresponding nondegenerate quadratic form $Q$ and let $\mu:\GO^\pm(V)\to \F_q^\times$ be defined by $Q(gv) = \mu(g)Q(v)$. In these cases, the role played by the determinant is played by $\mu$. Thus there is an identity \[
        |C_G(v)| = |C_\Omega(v)|\cdot|\mu(G)\cap \mu(C_\Delta(v))|,
    \]
    with $\Omega= \Sp(V)$, $\Delta = \GSp(V)$ in type $C_\ell$, $\Omega = \Omega^\pm(V)$, $\Delta = \GO^\pm(V)$ in types $D_\ell$, ${}^2 D_\ell$. Then the unipotent centralizer formulas for $\Omega$, together with $\mu(C_\Delta(u)) \le \mu(C_\Delta(v))$ for $u$ regular unipotent, give $|C_G(u)| \mid |C_G(v)|$.
\end{proof}

Recall that an element $x$ of $\sG$ is said to be \emph{regular} if $\dim C_{\sG}(x) = \rank(\sG)$, where $\rank (\sG)$ is the dimension of $\sT$, a maximal torus of $\sG$.

\begin{lemma}\label{lemma:ex-of-reg-unip-elts}
    The group $G$ contains a regular unipotent element.
\end{lemma}

\begin{proof}
    This is immediate from \cite[Prop 14.23]{DigneMichel1991}.
\end{proof}

\begin{lemma}\label{lemma:exc-unip-reg-divis-1}
    Suppose that $G$ belongs to one of the exceptional families under consideration. Then there exists a nontrivial unipotent element $x\in G$ such that $|C_G(x)|$ divides the centralizer order of all other nontrivial unipotent elements of $G$.
\end{lemma}

\begin{proof}
    By Lemma~\ref{lemma:ex-of-reg-unip-elts}, the group $G$ contains a regular unipotent element. Write $q=2^f$.

    First suppose that $G$ has type $G_2$, $F_4$, $E_6$, ${}^2E_6$, $E_7$, $E_8$, or ${}^2F_4$. Tables~22.2.1--22.2.6 of \cite{liebeckSeitz2012} give the following information. In type $G_2$, a regular unipotent centralizer has order $2q^2$, while every other nonidentity unipotent centralizer has $2$-adic valuation at least $4f$. Since $v_2(2q^2)=1+2f\leq 4f$, the required divisibility follows. In type $F_4$, the corresponding order is $4q^4$, while every other centralizer has $2$-adic valuation at least $1+6f$. Since $v_2(4q^4)=2+4f\leq 1+6f$, the result follows. In types $E_6$ and ${}^2E_6$, the corresponding order is $2q^6$, while every other centralizer has $2$-adic valuation at least $8f$. Since $v_2(2q^6)=1+6f\leq 8f$, the result follows. In type $E_7$, the corresponding order is $4q^7$, while every other centralizer has $2$-adic valuation at least $1+9f$. Since $v_2(4q^7)=2+7f\leq 1+9f$, the result follows. In type $E_8$, the corresponding order is $4q^8$, while every other centralizer has $2$-adic valuation at least $2+10f$. Since $v_2(4q^8)=2+8f\leq 2+10f$, the result follows. Finally, in type ${}^2F_4$, the corresponding order is $4q^2$, while every other centralizer has $2$-adic valuation at least $1+3f$. Since $v_2(4q^2)=2+2f\leq 1+3f$, the result follows.

    Suppose next that $G={}^3D_4(q)$. By \cite[\S 0.5]{spaltenstein1982}, the nonidentity unipotent centralizer orders in characteristic $2$ are $q^{12}(q^6-1)$, $q^{10}(q^2-1)$, $2(q^2+q+1)q^8$, $2(q^2-q+1)q^8$, $q^6$, $2q^4$, and $2q^4$. The final two orders correspond to the two regular unipotent classes. Since $q$ is even, $2q^4$ divides every order in this list.

    Finally, suppose that $G={}^2B_2(q)$. By \cite[\S\S 2,12]{suzuki1962}, the nonidentity unipotent centralizer orders are $q^2$ and $2q$, with the latter afforded by regular unipotent elements. Since $q$ is even, $2q\mid q^2$. The result follows.

\end{proof}

\section{Some Semisimple Elements in the Unipotent Component} \label{sec:4}

We have shown that $\cD(G)$ contains a distinguished component $\cU$, the unipotent component, which contains all nonidentity unipotent class sizes. In this section, we show that several types of semisimple elements also lie in $\cU$. Proposition~\ref{prop:noncent-nonreg-ss--in-U} shows that noncentral nonregular semisimple elements lie in $\cU$. Proposition~\ref{prop:noncent-ss-|C|-even--in-U}, which is an easy consequence of Lemmas~\ref{lemma:2-1} and~\ref{lemma:2-2}, shows that noncentral semisimple elements with even centralizer order lie in $\cU$.

\begin{lemma}\label{lemma:nr-ss--unip}
    Let $s\in G$ be a nonregular semisimple element. Then $C_G(s)$ contains a nontrivial unipotent element.
\end{lemma}

\begin{proof}
    Write $\sC = C_\sG(s)^\circ$, which is reductive by \cite[Thm. 14.2]{MalleTesterman2011}. Note that $F(C_\sG(s)) = C_\sG(F(s)) = C_\sG(s)$, so $\sC$ is $F$-stable. Moreover, $\sC$ is not a torus. Indeed, by \cite[Prop. 26.6]{MalleTesterman2011}, there exists an $F$-stable maximal torus $\sT$ of $\sG$ containing $s$, with $\sT$ contained in $\sC$ by \cite[Prop. 14.1]{MalleTesterman2011}. If $\sC$ were a torus, then the equality $\sT = \sC$ would follow, and therefore $\dim C_\sG(s) = \dim \sC = \dim \sT = \rank \sG$, contradicting $s$ being nonregular. By \cite[Cor. 21.12]{MalleTesterman2011}, we may choose $\sT \le \sB \le \sC$, with $\sT$ an $F$-stable maximal torus in $\sC$, and $\sB$ an $F$-stable Borel subgroup.

    Let $\Phi$ be the root system of $\sC$ with respect to $\sT$, and let $\Phi^+$ be the positive system determined by $\sB$. Since $\sC$ is not a torus, $\Phi^+ \ne \varnothing$ \cite[Thm. 11.1]{MalleTesterman2011}. Let $\sU = R_u(\sB)$, which is unipotent, $F$-stable, and normalized by $\sT$. Moreover, $\sU^F$ is nontrivial, since, by \cite[Prop. 23.8]{MalleTesterman2011}, \[
        |\sU^F| = \prod_{\alpha \in \Phi^+} q_\alpha > 1,
    \]
    where the $q_\alpha$ are positive powers of $2$. Choose $1\ne u\in \sU^F$, then \[
        u\in \sC^F \le C_\sG(s)^F = C_G(s),
    \]
as claimed.
\end{proof}

\begin{prop}\label{prop:noncent-nonreg-ss--in-U}
    Every noncentral nonregular semisimple element of $G$ has class size lying in the unipotent component $\cU$ of $\cD(G)$.
\end{prop}

\begin{proof}
    Let $s\in G\setminus Z(G)$ be semisimple and nonregular. Then by Lemma~\ref{lemma:nr-ss--unip}, there exists a nonidentity unipotent element $u\in C_G(s)$. By Lemma~\ref{lemma:2-1}, the order of $s$ is odd and the order of $u$ is a power of two. By Lemma~\ref{lemma:2-2}, the order of $Z(G)$ is odd, so $u\notin Z(G)$. Thus Lemma~\ref{lemma:4-AIN} gives $s\sim u$, and by Proposition~\ref{prop:unip-component}, $u\in \cU$, so also $s\in \cU$.
\end{proof}

\begin{prop}\label{prop:noncent-ss-|C|-even--in-U}
   Let $s\in G$ be a noncentral semisimple element with $|C_G(s)|$ even. Then $s\in \cU$.
\end{prop}

\begin{proof}
    Since the order of $C_G(s)$ is even, it contains an element $u$ of order $2$, which is unipotent and noncentral by Lemmas~\ref{lemma:2-1} and~\ref{lemma:2-2}. Since $|s|$ is odd by Lemma~\ref{lemma:2-1}, we have $(|s|, |u|) = 1$ and $su = us$, and thus $s\sim u\in \cU$ by Lemma~\ref{lemma:4-AIN}.
\end{proof}

Recall that an \emph{isogeny} is a map $\varphi: \sX \to \sY$ of linear algebraic groups which is surjective and has finite kernel. A semisimple linear algebraic group $\sG$ with root system $\Phi$ is said to be \emph{simply connected} if $X(\sT)=\Omega$, equivalently if $\Lambda(\sG)=1$, and it is said to be of \emph{adjoint type} if $X(\sT)=\Z\Phi$. Let $\sG_{\text{sc}}(\Phi)$ and $\sG_{\text{ad}}(\Phi)$ denote the isomorphism classes of simply-connected and adjoint-type linear algebraic groups with root system $\Phi$, respectively. Then by \cite[\S 9.1]{MalleTesterman2011}, there are natural isogenies \[
    \sG_{\text{sc}}(\Phi) \to \sG \to \sG_{\text{ad}}(\Phi).
\]
We say that an isogeny $\rho:\sX\to \sY$ is \emph{central} if its kernel is contained in the center of $\sX$. Let $\sG_{\mathrm{sc}}$ denote the connected, simply-connected, semisimple linear algebraic group of the same root system as $\sG$, and let $\pi:\sG_{\mathrm{sc}}\to \sG$ be the simply-connected covering isogeny. For a linear algebraic group $\sH$, the \emph{component group} is $\sH / \sH^\circ$.

\begin{lemma}\label{lemma:semisimple-component-group}
Let $t\in \sG$ be a semisimple element such that $t^r \in Z(\sG)$ for some prime $r$. Then the following hold.
\begin{itemize}
\item[(1)] The {component group} $C_\sG(t) / C_\sG(t)^\circ$ is isomorphic to a quotient of a subgroup of $\ker \pi$, and has exponent dividing $r$.
\item[(2)] If this component group has nontrivial $r$-part, then $r$ divides the order of the fundamental group of $\sG$.
\end{itemize}
\end{lemma}

\begin{proof}
By \cite[Prop. 14.20]{MalleTesterman2011}, the component group $C_\sG(t)/C_\sG(t)^\circ$ is isomorphic to a quotient of a subgroup of $\ker\pi$.

Choose a lift $\hat t\in \sG_{\mathrm{sc}}$ with $\pi(\hat t)=t$. Set $\sE=\pi^{-1}(C_\sG(t))$, and define ${\theta:\sE\to \ker\pi}$ by $\hat g\mapsto [\hat g,\hat t]$. It is clear that $\theta$ is well-defined, as $\pi([\hat t, \hat g]) = [g,t] = 1$, and that $\theta$ is a group homomorphism since $\ker\pi$ is central. Since $\ker\theta=C_{\sG_{\mathrm{sc}}}(\hat t)$, we have $\theta(\sE) \cong  \sE/C_{\sG_{\mathrm{sc}}}(\hat t)\hookrightarrow \ker\pi$.

For $\hat g\in \sE$, since $\pi(\hat g)\in C_\sG(t)$, we have $\pi([\hat g,\hat t])=1$, thus $[\hat g,\hat t]\in \ker\pi$, which is central, so $[\hat g,\hat t]^r=[\hat g,\hat t^r]$. Since $\pi(\hat t^r)=t^r\in Z(\sG)$, for every $\hat x\in \sG_{\mathrm{sc}}$ we have $[\hat t^r,\hat x]\in \ker\pi$. The map $\hat x\mapsto [\hat t^r,\hat x]$ is a map from the connected group $\sG_{\mathrm{sc}}$ to $\ker\pi$, which is finite. Therefore the map is constant, and evaluating at $\hat x=1$ shows that it is the trivial map. Thus $\hat t^r\in Z(\sG_{\mathrm{sc}})$. Therefore $[\hat g,\hat t]^r=1$, so every element of $\theta(\sE)$ has order dividing $r$. By the identification in \cite[Prop. 14.20]{MalleTesterman2011}, $C_\sG(t)/C_\sG(t)^\circ$ has exponent dividing $r$.

Lastly, since $C_\sG(t)/C_\sG(t)^\circ$ is isomorphic to a quotient of a subgroup of $\ker\pi$, if this component group has nontrivial $r$-part, then $r$ divides $|\ker\pi|$. But the order of $\ker\pi$ divides the order of the fundamental group of $\sG$.
\end{proof}

Now we focus on the remaining semisimple elements of $G$ whose components have not yet been determined. Fix a noncentral element $s\in G$ which is regular and semisimple. When necessary, we also assume that its centralizer has odd order. We fix the notation $\sT=C_\sG(s)^\circ$, which by \cite[Cor. 24.10]{MalleTesterman2011} is a maximal torus, and set $T=\sT^F$. Since both $Z(G)\le C_G(s)$ and $T\le C_G(s)$, set $A=TZ(G)\le C_G(s)$. In Section~\ref{sec:5}, we treat the case $A<C_G(s)$. In Section~\ref{sec:6}, we handle the case $A=C_G(s)$. These sections identify the elements lying in $\cU$, the isolated vertices, and the exceptional two-vertex components.

\section{Regular Semisimple Elements with $A < C_G(s)$}\label{sec:5}
\begin{lemma}\label{lemma:5-2}
    The center $Z(G)$ of $G$ is contained in the center $Z(\sG)$ of $\sG$.
\end{lemma}

\begin{proof}
    By \cite[Cor. 24.13]{MalleTesterman2011}, $Z(G) = Z(\sG^F) = Z(\sG)^F \le Z(\sG)$.
\end{proof}

Fix an element $t = s^m$, for some $m\ge 1$, and assume that $t$ is also noncentral, regular, and semisimple. Then similarly $C_\sG(t)^\circ = \sT$, and we define \[
    C_G(t)^\circ = C_G(t) \cap C_\sG(t)^\circ.
\]
Clearly $C_\sG(s)\le C_\sG(t)$, and as $\sT$ is connected, $\sT  = C_\sG(s)^\circ \le C_\sG(t)$. Both $C_\sG(s)^\circ$ and $C_\sG(t)^\circ$ are maximal tori, as $s$ and $t$ are regular semisimple, and thus must be equal. Since $\sT \le C_\sG(t)$, it follows that $C_G(t)^\circ = T = C_G(s)^\circ$. Because $t = s^m$, it follows that $C_G(s) \le C_G(t)$, so if $C_G(s) > A$, it also follows that $C_G(t) > A$. Henceforth, let $\overline G = G / Z(G)$, and for $g\in G$, let $\overline g$ denote its image in $\overline G$.

\begin{lemma}\label{lemma:ss-odd-nonreg}
    Assume that $C_G(s) > A$ and $|C_G(s)|$ is odd. Then one of the following holds.
    \begin{itemize}
        \item[(1)] The element $s$ belongs to the unipotent component $\cU$.
        \item[(2)] There exists a power $t = s^m$ and an odd prime $r$ such that $t$ is noncentral, regular, and semisimple, with $|C_G(t)|$ odd and $C_G(t) > A$, and moreover $\overline t$ has order $r$ in $\overline G = G/Z(G)$, and $r$ divides the order of the fundamental group $\Lambda(\sG)$. 
    \end{itemize}
\end{lemma}

\begin{proof}
    Since $\overline s \in \overline G$ is nontrivial by assumption, we may choose a power $t = s^m$ such that $\overline t\in \overline G$ has prime order $r$. By Lemma~\ref{lemma:4-AIN}, $s\sim t$, and by Lemma~\ref{lemma:2-1}, $r$ is odd since $t$ is semisimple. If $t$ is nonregular, then $t\in \cU$ and therefore $s\in \cU$, by Proposition~\ref{prop:noncent-nonreg-ss--in-U}, and the result is proved. Therefore, assume that $t$ is regular. If $|C_G(t)|$ is even, then by Proposition~\ref{prop:noncent-ss-|C|-even--in-U}, $t\in \cU$, and therefore also $s\in \cU$. So assume that $|C_G(t)|$ is odd.

    Next consider the map \[
        p: C_G(t) \to C_\sG(t) / C_\sG(t)^\circ 
    \]
    to the component group. The kernel is $C_G(t) \cap C_\sG(t)^\circ = T$. By Lemma~\ref{lemma:5-2}, $Z(G)\le Z(\sG)$, and moreover $Z(\sG) \le C_\sG(\sT) = \sT$, where the equality follows from \cite{MalleTesterman2011}, using that $\sG$ is connected and reductive. Then $Z(G) \le \sT\cap G = \sT^F = T$, so $A = TZ(G) = T$, and therefore $C_G(t) > T$. Thus the image of $p$ is nontrivial, as $T$ is its kernel.

    Since $t^r \in Z(G) \le Z(\sG)$, Lemma~\ref{lemma:semisimple-component-group} implies that $C_\sG(t) / C_\sG(t)^\circ$ has exponent dividing $r$. The nontrivial image of $C_G(t)$ therefore does as well, and $r$ divides the order of the fundamental group of $\sG$.
\end{proof}

\begin{lemma}\label{lemma:four-cases}
    Assume the hypotheses of Lemma~\ref{lemma:ss-odd-nonreg}. Then the second consequence can occur only if the Lie type of $\sG$ is one of the following: $A_\ell$, ${}^2A_\ell$, $E_6$, or ${}^2 E_6$.
\end{lemma}

\begin{proof}
    In this situation, an odd prime $r$ divides the order of $\Lambda(\sG)$. By inspection of \cite[Table 9.2]{MalleTesterman2011}, the only Lie types whose fundamental group order admits an odd prime divisor are those listed. Indeed, for root systems $\Phi=A_{n-1}$ and $\Phi=E_6$, the corresponding fundamental groups are respectively $\Lambda(\Phi) = \Z/n$ and $\Lambda(\Phi) = \Z/3$.
\end{proof}

\begin{lemma}\label{lemma:two-are-impossible-sc}
     Assume the hypotheses of Lemma~\ref{lemma:ss-odd-nonreg}, and assume further that $\sG$ is simply-connected. Then the second consequence is impossible.
\end{lemma}

\begin{proof}
    Let $t\in G$ be the regular semisimple element in the second case of Lemma~\ref{lemma:ss-odd-nonreg}. Then by Lemma~\ref{lemma:semisimple-component-group}, the quotient $C_\sG(t) / C_\sG(t)^\circ$ is a quotient of a subgroup of $\ker \pi$, where $\pi$ is the natural simply-connected isogeny. Since $\sG$ is simply-connected, $\ker\pi$ is trivial, so $C_\sG(t) = C_\sG(t)^\circ$. Since $t$ is regular and semisimple, $\sT = C_\sG(t)^\circ = C_\sG(t)$. But also $Z(G)\le C_G(t)=(C_\sG(t))^F=\sT^F=T$, and therefore $A=TZ(G)=T$, a contradiction to the assumption that $C_G(t) > A$.
\end{proof}

\begin{lemma}\label{lemma:E6-adjoint--impossible}
    Assume the hypotheses of Lemma~\ref{lemma:ss-odd-nonreg}, and assume further that $\sG$ is of adjoint type $E_6$ or ${}^2E_6$.
    Then the second consequence of the lemma is impossible.
\end{lemma}

\begin{proof}
    Since the fundamental group of $\sG$ is $\Z/3$, the prime $r$ is $3$, so $t^3 = 1$. We will show that a nontrivial semisimple element of order $3$ in $\sG$ is not regular. Let $\sT$ be a maximal torus containing $t$, and let $\Phi$ denote the root system of $\sG$ with respect to $\sT$. Since $\sG$ is adjoint, $X(\sT)$ is the root lattice $\Z\Phi$. Evaluation at $t$ defines a map \[
        f:X(\sT) = \Z\Phi \to \mu_3.
    \]
    Write the images $\alpha_i(t) = \zeta^{e_i}$, where $\zeta$ is a primitive $3$rd root of unity, $e_i\in \F_3$, and $\alpha_i$ are the six simple roots. Take the Dynkin diagram to be as in \cite[Table 9.1]{MalleTesterman2011}, in which $\alpha_4$-$\alpha_5$-$\alpha_6$ forms a subdiagram $A_3$. By \cite[A.25]{MalleTesterman2011}, the roots of $A_3$ arising from this subdiagram are also roots of $E_6$, namely $\alpha_4$, $\alpha_5$, $\alpha_6$, $\alpha_4+\alpha_5$, $\alpha_5+\alpha_6$, and $\alpha_4+\alpha_5+\alpha_6$.

    By \cite[Cor. 14.10]{MalleTesterman2011}, $t$ is regular if and only if $\alpha(t) \ne 1$ for all roots $\alpha$. Assuming $t$ is regular, $e_4$, $e_5$, and $e_6$ are nonzero. Also $e_4 + e_5 \ne 0 \implies e_4 = e_5$, and similarly $e_5 = e_6$. Then $e_4 + e_5 + e_6 = 0$, a contradiction.
\end{proof}

\begin{lemma} \label{lemma:An-adjoint}
    Assume the hypotheses of Lemma~\ref{lemma:ss-odd-nonreg}, and suppose further that $\sG$ is of Lie type $A_{n-1}$; thus $G = \sG^F$ has type $A_{n-1}$ in the split case, or ${}^2A_{n-1}$ in the unitary case. Suppose that the second consequence of Lemma~\ref{lemma:ss-odd-nonreg} holds. Then $n = r$ is an odd prime, $\sG$ is of adjoint type, and $G = \PGL_r(q)$ in the split case or $G = \PGU_r(q)$ in the unitary case. Moreover, in the split case $r\mid q - 1$ and in the unitary case $r\mid q + 1$.
\end{lemma}

\begin{proof}
    Let $\pi : \sG_{\text{sc}} \to \sG$ be the simply-connected cover. In this case, $\sG_{\text{sc}}\cong \SL_n = \SL_n(k)$. By \cite[\S 9.1]{MalleTesterman2011}, since we have taken $\sG$ to be connected, $\ker\pi \le Z(\SL_n)$. Recalling our assumptions, $C_G(t) > A$, $t$ is regular and semisimple, and thus $C_\sG(t)^\circ = \sT$; the kernel of the map of $C_G(t)$ to the component group is $T = \sT^F$, $Z(G)\le T$, and thus $A = T$. So $C_G(t)>T$ gives a nontrivial subgroup of the component group, which is a quotient of a subgroup of $\ker\pi$, with exponent dividing $r$, by Lemma~\ref{lemma:semisimple-component-group}. Therefore, $r\mid|\ker\pi|$. Secondly, since $\Lambda(\sG) = \Z/n$, $|\ker\pi|\mid n$, and together we get $r\mid n$.

    Lift $t$ to $\hat t \in \SL_n$. Since $t^r\in Z(\sG)$, also $\hat t^r \in Z(\SL_n)$, which is composed of scalars. Here we have used the fact that $\pi^{-1}(Z(\sG)) = Z(\SL_n)$. The minimal polynomial of $\hat t$ divides $X^r - \lambda$ for some scalar $\lambda$, so $\hat t$ has at most $r$ eigenvalues.

    We show now that $\hat t$ has $n$ distinct eigenvalues. Let $\widehat \sT = \pi^{-1}(\sT)^\circ$, and choose the lift $\hat t$ to be in $\widehat\sT$. Note that $\widehat\sT$ is a maximal torus of $\SL_n$, which without loss of generality we can take to be the diagonal torus. Therefore we write $\hat t = \diag(\lambda_1, \dots, \lambda_n)$ with $\prod \lambda_i = 1$. The roots of $\sG$ with respect to $\sT$ pull back to roots of $\SL_n$ with respect to $\widehat \sT$, which by \cite[8.2]{MalleTesterman2011} are of the form $\chi_{ij}$ for $1\le i < j \le n$, where $\chi_{ij}(\diag(\lambda_1, \dots, \lambda_n)) = \lambda_i\lambda_j^{-1}$. The element $t$ is regular if and only if $\alpha(t) \ne 1$ for all roots $\alpha$ of $\sT$, and hence for all roots of $\widehat \sT$, which arise as the pullback of roots of $\sT$. Thus $\lambda_i \ne \lambda_j$ for all $i\ne j$, which provides $n$ distinct eigenvalues. 

    Since $n\le r$ and $r\mid n$, we get $n = r$, and hence $\ker\pi = Z(\SL_n)$. This implies that $\sG$ is of adjoint type. So $G = \PGL_r(q)$ in the split case and $G = \PGU_r(q)$ in the unitary case. 
    
    It remains to show that $r\mid q- 1$ in the split case and $r\mid q + 1$ in the unitary case. Write $T = (C_\sG(t)^\circ)^F$, and since $C_G(t) > T$, take $g\in C_G(t) \setminus T$. Choose a lift $\hat g\in \SL_r$. Then $[\hat g, \hat t] \in \ker \pi = \mu_r$, which consists of scalars, so write $[\hat g, \hat t] = c$. Suppose $c = 1$. Then $\hat g\in C_{\SL_r}(\hat t)$, the maximal torus lying over $T$, and thus $g\in T$, a contradiction. So $c$ has order $r$. Since $g$ and $t$ are fixed by $F$, their preimages are fixed modulo scalars, so write $F(\hat g) = a\hat g$ and $F(\hat t) = b\hat t$, where $a$ and $b$ are scalars. Then $F(\hat g \hat t \hat g^{-1}) = b\hat g\hat t\hat g^{-1} = bc\hat t = cF(\hat t)$. But this expression also equals $F(c\hat t) = F(c)F(\hat t)$, so $F(c) = c$. Due to \cite[21.14]{MalleTesterman2011}, in the split case $F(c) = c^q$ and in the unitary case $F(c) = c^{-q}$, and the claim follows.
\end{proof}

\begin{lemma}\label{lemma:a-greater-r}
    Assume the hypotheses of Lemma~\ref{lemma:An-adjoint}. Let $\varepsilon = +1$ in the split case and $\varepsilon = -1$ in the unitary case, set $a = q - \varepsilon$. Let $G = \PGL_r(q)$ in the split case and $G = \PGU_r(q)$ in the unitary case. If $a > r$, then one of the following holds.
    \begin{itemize}
        \item[(1)] $F$ fixes the $r$ eigenlines of a lift of $t$, and $t \in \cU$.
        \item[(2)] $F$ acts transitively on these eigenlines, \[
            |T| = \frac{q^r - \varepsilon ^r}{q - \varepsilon},
        \]
        $|C_G(t) : T| = r$, and there exists $y\in T$ such that $C_G(y) = T$. Moreover, $y\sim t$.
        
    \end{itemize}
\end{lemma}

\begin{proof}
    By Lemma~\ref{lemma:An-adjoint}, $r\mid a$, so we may write $a = rm$, with $m > 1$. Since $a = q - \varepsilon$ is odd, $m$ is at least $3$. Consider the cyclic group $K = \langle \omega \rangle$ of order $a$, where $K = \F_q^\times$ in the split case and $K = \{x\in \F_{q^2}^\times \mid x^{q+1} = 1\}$ in the unitary case. Then $\langle \omega^m\rangle \le \langle \omega\rangle$ is a subgroup of order $r$. 

        Let
    \[
        \pi:\SL_r(k)\longrightarrow \sG
    \]
    be the simply connected covering isogeny, and choose
    $\hat t\in\SL_r(k)$ such that $\pi(\hat t)=t$. By the preceding discussion,
    $\hat t^r$ is scalar and $\hat t$ has $r$ distinct eigenvalues,
    which may be written as $\lambda\mu_r$ for some
    $\lambda\in k^\times$. Since $r\mid a$, we have
    $\mu_r=\langle\omega^m\rangle\le K$. Take $h\in C_G(t)$ and choose a lift $\hat h\in\SL_r(k)$. Then $\hat h\hat t\hat h^{-1}=c_h\hat t$ for some scalar $c_h \in\mu_r$. The scalar $c_h$ is independent of the choices of lifts, so $ h\longmapsto c_h$ defines a homomorphism $C_G(t)\to\mu_r$. Its kernel is $T$. Indeed, $c_h=1$ if and only if $\hat h$ centralizes $\hat t$, and the centralizer of $\hat t$ in
    $\SL_r(k)$ is the maximal torus lying above $\sT$. Therefore
    \[
        C_G(t)/T\hookrightarrow\mu_r.
    \]
    Since $C_G(t)>T$, this quotient is nontrivial, and hence $C_G(t)/T\cong C_r$ and $|C_G(t)|=r|T|$.

    By the preceding discussion, $F$ acts on the eigenlines of
    $\hat t$ either trivially or as an $r$-cycle.

    Suppose first that $F$ fixes every eigenline of $\hat t$.     Let $\sT_0$ be the standard diagonal $F$-stable maximal torus of $\sG$. We may choose $g\in\sG$ such that $\sT=g\sT_0g^{-1}$. Since $\sT$ is $F$-stable, it follows that $g^{-1}F(g)\in N_{\sG}(\sT_0)$. Under the identification $N_{\sG}(\sT_0)/\sT_0\cong S_r$ of the Weyl group, the image of $g^{-1}F(g)$ is the permutation induced by $F$ on the eigenlines of $\hat t$. Using that $F$ fixes every eigenline, it follows that $g^{-1}F(g)\in\sT_0$. Since $\sT_0$ is connected, the Lang--Steinberg Theorem gives an element $h\in\sT_0$ such that
    \[
        h^{-1}F(h)=\bigl(g^{-1}F(g)\bigr)^{-1}.
    \]
    It follows that $gh\in G$ and
    \[
        \sT=(gh)\sT_0(gh)^{-1}.
    \]
    Thus, after conjugating by an element of $G$, we may assume that
    $T=\sT^F$ is the standard diagonal split torus in the split case
    and the standard diagonal unitary torus in the unitary case. Now define
    \[
        E=\{2m,3m,\dots,(r-1)m,1,m-1\}\subseteq\mathbb Z/a\mathbb Z,
    \]
    which has $r$ distinct elements. Let $\hat y$ be the diagonal
    matrix with diagonal entries $\omega^e$, for $e\in E$, in any
    order. Then $\hat y\in\GL_r(q)$ in the split case and
    $\hat y\in\GU_r(q)$ in the unitary case, and the image $y$ of $\hat y$ in $G$ belongs to $T$.

    The matrix $\hat y$ has $r$ distinct eigenvalues, so $y$ is
    regular and semisimple. We claim that $C_G(y)=T$. Take
    $h\in C_G(y)$, and choose a lift $\hat h\in\GL_r(q)$ in the
    split case or $\hat h\in\GU_r(q)$ in the unitary case. Then $\hat h\hat y\hat h^{-1}=d\hat y$ for some scalar $d\in K$. Write $d=\omega^\delta$. Since multiplication by $d$ preserves the eigenvalue set, we have $E+\delta=E$. If $\delta\ne0$, translation by $\delta$ has orbits of length dividing $|E|=r$. Since a nonzero translation has no fixed points and $r$ is prime, $\delta$ has order $r$ and $E$ is a coset of the unique subgroup $\langle m\rangle$ of order $r$. This is impossible, since the elements of such a coset are all congruent modulo $m$, whereas $E$ contains both $1$ and $m-1$. Thus $\delta=0$, so $d=1$. Since $\hat y$ has distinct
    eigenvalues, $\hat h$ is diagonal. Hence $C_G(y)=T$.

    Let $z\in T$ be the image of $\operatorname{diag}(\omega,\omega,1,\dots,1)$. Then $z$ is noncentral, semisimple, and nonregular. Since
    \[
        |C_G(y)|=|T|\mid |C_G(t)|,
    \]
    we have $t\sim y$. Also,
    \[
        T=C_G(y)\le C_G(z),
    \]
    so $y\sim z$. By
    Proposition~\ref{prop:noncent-nonreg-ss--in-U}, we have
    $z\in\cU$, and therefore $t\in\cU$.

    Suppose now that $F$ acts as an $r$-cycle on the eigenlines of
    $\hat t$. Then $T$ is the irreducible rational torus, and
    \[
        |T|=\frac{q^r-\varepsilon^r}{q-\varepsilon},
    \]
    by \cite[Prop.~25.3(c),~Ex.~25.8]{MalleTesterman2011}. Choose a generator $y$ of the cyclic group $T$. Since the torus
    is irreducible and $r$ is prime, $y$ is regular semisimple.
    Hence $C_G(y)\le N_G(T)$. Moreover, the Weyl group is
    \[
        N_G(T)/T\cong C_r,
    \]
    and every nonidentity element of this quotient acts
    nontrivially on a generator of $T$. It follows that
    \[
        C_G(y)=T.
    \]
    Therefore
    \[
        |C_G(y)|=|T|\mid r|T|=|C_G(t)|,
    \]
    and hence $t\sim y$.
\end{proof}

\begin{lemma}\label{lemma:5-8}
    Assume the hypotheses of Lemma~\ref{lemma:An-adjoint}, and
    retain the notation of Lemma~\ref{lemma:a-greater-r}. Suppose that $a=r$ and
    that $F$ fixes every eigenline of $\hat t$. Then
    $|C_G(t)|=r^r$, and the associated vertex in $\cD(G)$ is
    \emph{isolated}.
\end{lemma}

\begin{proof}
    Since $F$ fixes every eigenline of $\hat t$, the same argument as in the proof of Lemma~\ref{lemma:a-greater-r} allows
    us to assume that $T$ is the standard diagonal torus. Taking the notation of that proof, $|K|=r$, and, since $\sG$ is of adjoint type, $ T\cong K^r/K$. Hence $|T|=r^{r-1}$. Moreover, we have seen that $C_G(t)/T\hookrightarrow\mu_r$.
    Since $C_G(t)>T$, it follows that $C_G(t)/T\cong C_r$. Let $w$ be the image of the permutation matrix which cyclically permutes the $r$ eigenlines of $\hat t$. Then $w$ centralizes $t$ projectively, and the same scalar-commutator calculation gives $C_G(t)=T\rtimes\langle w\rangle$. Therefore $|C_G(t)|=r|T|=r^r$.

    Set $P=C_G(t)$. We claim that $P$ is a Sylow $r$-subgroup. This follows from computing $v_r(|G|) = r$ from the formula\[
        |G| = q^{r(r-1)/2}\prod_{i = 2}^r (q^i - \varepsilon^i).
    \]
    Indeed, since $q-\varepsilon=r$, we have $ v_r(q^i-\varepsilon^i)=1+v_r(i)$ for $2\le i\le r$, and hence
    \[
        v_r(|G|)=\sum_{i=2}^r v_r(q^i-\varepsilon^i)
        =\sum_{i=2}^r (1+v_r(i))=(r-1)+1=r.
    \]

    Now we will argue that $t$ affords an isolated vertex. Take $x\in G\setminus Z(G)$ with $|x^G| \ne |t^G|$, and suppose that the vertex afforded by $x$ is adjacent to the vertex afforded by $t$.

    In the first case, we suppose that $|C_G(x)| \mid |P|$, in which case $C_G(x)$ is an $r$-group, and hence $|x|$ is odd. By Lemma~\ref{lemma:2-1}, the unipotent part $x_u$ is trivial, so $x$ is semisimple. If $x$ is not regular, then by Lemma~\ref{lemma:nr-ss--unip}, there is a nontrivial unipotent element in $C_G(x)$, a contradiction since this would have order a power of $2$. Thus $x$ is regular. Set $\sS = C_\sG(x)^\circ$, an $F$-stable maximal torus. Moreover, $\sS^F\le C_G(x)$, so $\sS^F$ has $r$-power order. By the parametrization of $G^F$-classes of $F$-stable maximal tori by $\phi$-conjugacy classes in the Weyl group \cite[Prop.~25.1]{MalleTesterman2011}, together with the description in terms of cycle types \cite[Example~25.4(2)]{MalleTesterman2011}, the $F$-stable maximal tori in $\PGL_r(q)$ and $\PGU_r(q)$ are parametrized by partitions $r=n_1+\cdots+n_m$, with fixed-point order
    \[
        \frac{\prod_i(q^{n_i}-\varepsilon^{n_i})}{q-\varepsilon}.
    \]
    Since $q-\varepsilon=r$, if some $n_i>1$, then the factor
    \[
        \frac{q^{n_i}-\varepsilon^{n_i}}{q-\varepsilon}
    \]
    has a prime divisor different from $r$. Indeed, if $n_i<r$, then this quotient is congruent to $\varepsilon^{n_i-1}n_i\not\equiv 0\pmod r$, and is greater than $1$. If $n_i=r$, then the same calculation gives
    \[
        v_r\left(\frac{q^r-\varepsilon^r}{q-\varepsilon}\right)=1,
    \]
    while the quotient is greater than $r$. Thus $\sS$ is a diagonal torus. Again, the same style of argument shows that since $x$ is regular and semisimple in $\sS$, there are $r$ distinct projective eigenvalues, hence the relevant scalar group $K$ has $|K| = r$, and $|\sS^F| = |K^r / K| = r^{r-1}$, and thus $|C_G(x)| = r^r = |P|$. In this case, $x$ gives rise to the same vertex in $\cD(G)$ as $t$.

    Next, we suppose that $|P| \mid |C_G(x)|$. By previous discussion, write $P = T\rtimes \langle w\rangle$, where $T\cong K^r / K$, and $w$ acts on $T$ by cyclically permuting the diagonal entries. Now, $C_G(x)$ contains a Sylow $r$-subgroup of $G$, thus, conjugating $x$ if necessary, we may assume that $P\le C_G(x)$, and therefore $x\in C_G(P)$. Since $t\in P$, $C_G(P) \le C_G(t) = P$, and therefore elements of $C_G(P)$ have the form $aw^i$ for $a\in T$ and $0\le i<r$, and also note that $aw^i$ centralizes $T\le P$. As $T$ is abelian, this implies that $w^i$ centralizes $T$. But $w$ acts on $T$ faithfully, so $w^i$ acts on $T$ trivially if and only if $i = 0$, so $C_G(P)\le T$. Therefore $C_G(P) = C_T(P) = C_T(w)$.

    Now an element of $T$ may be represented by $(a_0,\dots, a_{r-1})\in K^r$, with equivalence given by scalars. Such a representative is a representative of an element centralizing $w$ exactly if it is of the form $(1,c,c^2,\dots, c^{r-1})$, and indeed it follows that $|C_T(w)| =r$. If $x$ is a noncentral element of $C_G(P) = C_T(w)$, then the parameter $c$ is not $1$. Thus, the same argument used earlier shows that $x$ has $r$ distinct projective eigenvalues, and $C_G(x) = P$. Thus, $x$ does not give rise to a vertex different from $t$. In either divisibility case, every vertex adjacent to the vertex afforded by $t$ is equal to it. Therefore the vertex afforded by $t$ is isolated.
\end{proof}

\begin{lemma}\label{lemma:irreducible-two-vertex}
    Assume the hypotheses and notation of
    Lemma~\ref{lemma:An-adjoint}, and suppose that $F$ acts as
    an $r$-cycle on the eigenlines of $\hat t$. Set
    \[
        N=\frac{q^r-\varepsilon^r}{q-\varepsilon}.
    \]
    Then $|C_G(t)|=rN$, and there exists $y\in G$ such that $|C_G(y)|=N$. Moreover,
    \[
        \left\{\frac{|G|}{rN},\frac{|G|}{N}\right\}
    \]
    is a connected component of $\cD(G)$.
\end{lemma}

\begin{proof}
    Since $F$ acts transitively on the $r$ eigenlines of $\hat t$,
    the corresponding rational maximal torus has cycle type $(r)$.
    Thus $T=\sT^F$ is cyclic of order
    \[
        |T|=N=\frac{q^r-\varepsilon^r}{q-\varepsilon}.
    \]

    As before, we get $ C_G(t)/T\hookrightarrow\mu_r$. Since $C_G(t)>T$, this quotient is nontrivial. Therefore $ C_G(t)/T\cong C_r$ and hence $|C_G(t)|=rN$.

    Let $y$ be a generator of the cyclic group $T$. Since $y\in T$, it is semisimple. We claim that $y$ is regular. Suppose that a lift $\hat y$ has equal eigenvalues on two distinct eigenlines. Applying successive powers of $F$ preserves this equality. Since $F$ acts as an $r$-cycle and $r$ is prime, it follows that all eigenvalues of $\hat y$ are equal. Thus $\hat y$ is scalar, so $y=1$ in $G$, contradicting the fact that $y$ generates $T$. Therefore $\hat y$ has $r$ distinct eigenvalues, and $y$ is regular semisimple.

    Moreover, $N_G(T)/T\cong C_r$. The nonidentity elements of this quotient act nontrivially on $T$, and hence cannot fix a generator of $T$. Since $y$ is regular, every element of $C_G(y)$ normalizes $T$. It follows
    that $C_G(y)=T$, so $|C_G(y)|=N$. Therefore the two vertices $\frac{|G|}{rN}$ and $\frac{|G|}{N}$ are distinct and adjacent.

    It remains to show that no other vertex is adjacent to either of them. Let $x\in G\setminus Z(G)$, and suppose that the vertex of $x$ is adjacent to one of these two vertices. Since $N\mid rN$, either $|C_G(x)|\mid rN$ or $N\mid |C_G(x)|$. 

    Suppose first that $ N\mid |C_G(x)|$.
    By Zsigmondy's theorem, there is a prime $\ell\mid N$ such that $\operatorname{ord}_{\ell}(q)=r$ in the linear case and $\operatorname{ord}_{\ell}(q)=2r$ in the unitary case. The exceptional cases in Zsigmondy's theorem do not occur because $r\ge5$. By Cauchy's theorem, $C_G(x)$ contains an element $z$ of order $\ell$.

    The element $z$ is regular semisimple and lies in an irreducible maximal torus $S$ of order $N$. No nonidentity element of $N_G(S)/S\cong C_r$ centralizes $z$, by the choice of $\ell$.
    Hence $ C_G(z)=S$. Since $x$ centralizes $z$, we have $x\in S$.

    We claim that $x$ is regular. Indeed, if a lift of $x$ had two equal eigenvalues, the transitivity of $F$ on the eigenlines and the fact that $r$ is prime would imply that all its eigenvalues were equal. Then $x$ would be scalar and hence central, contrary to the choice of $x$. Thus $x$ is regular semisimple.

    It follows that every element of $C_G(x)$ normalizes $S$, and
    therefore $S\le C_G(x)\le N_G(S)$. Since $N_G(S)/S\cong C_r$, it follows that $|C_G(x)|=N$ or $|C_G(x)|=rN$.

    Now suppose that $|C_G(x)|\mid rN$. Since $rN$ is odd, $C_G(x)$ has odd order. In particular, $x$ has odd order and is semisimple. If $x$ were nonregular, then Lemma~\ref{lemma:nr-ss--unip} would give a nonidentity unipotent element in $C_G(x)$, contradicting the oddness of
    $|C_G(x)|$. Thus $x$ is regular.

    Set $\sS=C_\sG(x)^\circ$. Then $\sS$ is an $F$-stable maximal torus and $ |\sS^F|\mid |C_G(x)|\mid rN$. Set $a=q-\varepsilon$. The torus $\sS$ corresponds to a
    partition $r=n_1+\cdots+n_k$, and
    \[
        |\sS^F|
        =
        \frac{\prod_{i=1}^k(q^{n_i}-\varepsilon^{n_i})}
             {q-\varepsilon}
        =
        a^{k-1}\prod_{i=1}^k f_{n_i},
    \]
    where
    \[
        f_n=\frac{q^n-\varepsilon^n}{q-\varepsilon}.
    \]

    We claim that $\gcd(f_n,N)=1$ for $1<n<r$. Set  $b=\varepsilon q$. Up to a sign, $f_n$ and $N=f_r$ are the
    quotients $\frac{b^n-1}{b-1}$ and $\frac{b^r-1}{b-1}$. Suppose that a prime $p$ divides both. If $p\nmid b-1$, then the order of $b$ modulo $p$ divides both $n$ and $r$, and hence divides $\gcd(n,r)=1$, a contradiction. If $p\mid b-1$, then $f_n\equiv n\pmod p$ and $f_r\equiv r\pmod p$, so $p\mid\gcd(n,r)=1$, again a contradiction. This proves the
    claim.

    Since $r\mid N$, it follows that $\gcd(f_n,rN)=1$ for $1<n<r$. Therefore no part $n_i$ can satisfy $1<n_i<r$, since
    $f_{n_i}>1$ divides $|\sS^F|$, while $|\sS^F|\mid rN$.

    Thus either the partition is $(r)$ or every $n_i=1$. In the
    latter case, $|\sS^F|=a^{r-1}$. We have $\gcd(a,N)=r$ and $v_r(N)=1$. If $a$ has a prime divisor different from $r$, then that prime does not divide $N$, so $a^{r-1}\nmid rN$. If instead $a=r^e$, then $v_r(a^{r-1})=e(r-1)>2=v_r(rN)$, since $r\ge5$. This is again impossible.

    Consequently, the partition must be $(r)$. Hence $\sS^F$ is an irreducible maximal torus of order $N$. Since $x$ is regular, $\sS^F\le C_G(x)\le N_G(\sS^F)$, and $ N_G(\sS^F)/\sS^F\cong C_r$.
    Therefore $|C_G(x)|=N$ or $|C_G(x)|=rN$.

    Thus every vertex adjacent to either $|G|/(rN)$ or $|G|/N$ is
    one of these two vertices. Hence
    \[
        \left\{\frac{|G|}{rN},\frac{|G|}{N}\right\}
    \]
    is a connected component of $\cD(G)$.
\end{proof}

The preceding lemmas of this section together prove the following.

\begin{prop}\label{prop:sec-5}
    Let $s\in G\setminus Z(G)$ be regular and semisimple, and suppose that $|C_G(s)|$ is odd and $C_G(s)>A$. Then one of the following holds.
    \begin{itemize}
        \item[(1)] The vertex $|s^G|$ lies in $\cU$.

        \item[(2)] Either $G=\PGL_r(q)$ and $r=q-1$, or $G=\PGU_r(q)$ and $r=q+1$, where $r$ is an odd prime. In this case, $|C_G(s)|=r^r$, and the vertex $|s^G|$ is isolated in $\cD(G)$.

        \item[(3)] We have $G=\PGL_r(q)$ or $G=\PGU_r(q)$, where $r$ is an odd prime dividing $q-\varepsilon$, with $\varepsilon=1$ in the linear case and $\varepsilon=-1$ in the unitary case. If $N=(q^r-\varepsilon^r)/(q-\varepsilon)$, then $|s^G|$ lies in the two-vertex component
        \[
            \left\{\frac{|G|}{rN},\frac{|G|}{N}\right\}
        \]
        of $\cD(G)$.
    \end{itemize}
\end{prop}

\begin{proof}
    By Lemma~\ref{lemma:ss-odd-nonreg}, either $|s^G|\in\cU$, or there exist a power $t=s^m$ and an odd prime $r$ such that $t$ is noncentral, regular, and semisimple, $|C_G(t)|$ is odd, $C_G(t)>A$, the element $tZ(G)$ has order $r$, and $r$ divides the order of the fundamental group of $\sG$. In the first case, conclusion~(1) holds, so assume that the second case occurs. By Lemmas~\ref{lemma:four-cases},~\ref{lemma:two-are-impossible-sc}, and~\ref{lemma:E6-adjoint--impossible}, together with Lemma~\ref{lemma:An-adjoint}, we have $G=\PGL_r(q)$ or $G=\PGU_r(q)$, where $r$ is an odd prime. Let $\varepsilon=1$ in the linear case and $\varepsilon=-1$ in the unitary case. Then $r\mid q-\varepsilon$. Set $a=q-\varepsilon$. Thus $a\ge r$. Since $t=s^m$, Lemma~\ref{lemma:4-AIN} gives $s\sim t$. It is therefore enough to determine the component containing the vertex afforded by $t$. The Frobenius endomorphism $F$ acts on the $r$ eigenlines of a lift $\hat t$ either trivially or as an $r$-cycle. Suppose first that $F$ acts as an $r$-cycle. Set $N=(q^r-\varepsilon^r)/(q-\varepsilon)$. By Lemma~\ref{lemma:irreducible-two-vertex}, the vertex afforded by $t$ lies in the component
    \[
        \left\{\frac{|G|}{rN},\frac{|G|}{N}\right\}.
    \]
    Since $s\sim t$, the vertex afforded by $s$ lies in the same component. Thus conclusion~(3) holds. Now suppose that $F$ fixes every eigenline of $\hat t$. If $a>r$, then Lemma~\ref{lemma:a-greater-r} gives $t\in\cU$. Since $s\sim t$, it follows that $s\in\cU$, so conclusion~(1) holds. It remains to consider the case $a=r$. By Lemma~\ref{lemma:5-8}, we have $|C_G(t)|=r^r$, and the vertex afforded by $t$ is isolated. Since $s\sim t$, the vertices afforded by $s$ and $t$ lie in the same component. Since the vertex afforded by $t$ is isolated, we must have $|s^G|=|t^G|$, and hence $|C_G(s)|=|C_G(t)|=r^r$. Since $a=q-\varepsilon=r$, we have $r=q-1$ in the linear case and $r=q+1$ in the unitary case. Thus conclusion~(2) holds.
\end{proof}

\section{Regular Semisimple Elements with $A=C_G(s)$}\label{sec:6}

In this section, we treat the remaining case of regular semisimple elements with $A=C_G(s)$. Throughout this section, write $Z=Z(G)$. Given a group $H$, a subgroup $T\le H$ is called a \emph{CC-subgroup} if $C_H(t)\le T$ for all $t\in T\setminus Z(H)$. In the setting of algebraic groups, if $\sT\le\sH$ is a torus, then $\sT$ is a \emph{CC-torus} if $C_\sH(t)=\sT$ for all $t\in\sT\setminus Z(\sH)$. This notion is due to \cite{williams1981primegraph}, in his work on prime graphs, and is used in \cite{AbdolghafourianIranmaneshNiemeyer2017LieType}, which was the source of inspiration for this paper and its methods. In light of the previous section, in which we considered the subgroup $A=TZ$, where $T=\sT^F$ for an $F$-stable maximal torus $\sT$ of $\sG$, we say that $A$ is a \emph{direct CC-torus} if $C_G(a)=A$ for all $a\in A\setminus Z$.

We also recall the \emph{prime element graph} $\Gamma(G)$ associated to a finite group $G$. Its vertices are the primes dividing $|G|$, and two distinct primes $p$ and $q$ are adjacent if and only if $G$ contains an element of order divisible by $pq$. The following facts are due to \cite{williams1981primegraph}.

\begin{lemma}[\cite{williams1981primegraph}]\label{lemma:williams}
    Suppose that $A=TZ>Z$ is a direct CC-torus, and that $|A/Z|$ is odd and coprime to $|Z|$. Set $\pi=\pi(|A/Z|)$. Then the following hold.
    \begin{itemize}
        \item[(1)] The subgroup $A/Z$ is a CC-subgroup of $G/Z$.
        \item[(2)] The set $\pi$ is a complete connected component of the prime graph $\Gamma(G/Z)$, and $2\notin\pi$.
        \item[(3)] Every nontrivial $\pi$-element of $G/Z$ is conjugate into $A/Z$.
    \end{itemize}
\end{lemma}

\begin{lemma}\label{lemma:aux}
    Suppose that $G$ has type $A_{n-1}$ or ${}^2A_{n-1}$, and let $a\in T$ be regular and semisimple. Suppose that $aZ$ has prime order $r$, and that $r\mid |Z|$ and $r\mid |T/Z|$. Then $n=r$, the group $\sG$ is simply connected, and $T$ is a split diagonal torus in the linear case and a diagonal unitary torus in the unitary case.
\end{lemma}

\begin{proof}
    Let $\pi:\sG_{\mathrm{sc}}=\SL_n\to\sG$ be the simply connected covering isogeny, and lift $a$ to $\hat a$. Since $a^r\in Z(\sG)$, it follows that $\hat a^r\in Z(\SL_n)$. Therefore $\hat a$ has at most $r$ eigenvalues, so $n\le r$. Since $r\mid |Z|$, we also have $r\mid |Z(\sG)|$. The group $Z(\sG)$ is a quotient of $Z(\SL_n)=\mu_n$, so $r\mid n$. Thus $n=r$.

    Since $\ker\pi\le\mu_r$ and $r$ is prime, the kernel is either trivial or equal to $\mu_r$. The latter case would give $\sG\cong\PGL_r$, which has trivial center, contradicting $r\mid |Z|$. Hence $\ker\pi=1$, so $\sG$ is simply connected.

    Set $\varepsilon=1$ in the split case and $\varepsilon=-1$ in the unitary case. Since $r\mid |Z|$, we have $r\mid q-\varepsilon$. The eigenvalues of $\hat a$ have the form $\lambda,\lambda\zeta,\ldots,\lambda\zeta^{r-1}$, where $\zeta$ has order $r$. The Steinberg endomorphism $F$ acts on the eigenlines either trivially or as an $r$-cycle. Suppose that $F$ acts as an $r$-cycle. Then
    \[
        |T|=\frac{q^r-\varepsilon^r}{q-\varepsilon}
    \]
    and $v_r(|T|)=1$. Since $G=\SL_r(q)$ or $G=\SU_r(q)$, we also have $v_r(|Z|)=1$, and therefore $r\nmid |T/Z|$, a contradiction. Hence every eigenline is $F$-stable, so $T$ is a split diagonal torus in the linear case and a diagonal unitary torus in the unitary case.
\end{proof}

\begin{lemma}\label{lemma:corpimality}
    Suppose that $A$ is a direct CC-torus of odd order. Then $(|A/Z|,|Z|)=1$.
\end{lemma}

\begin{proof}
    Suppose not, and let $r$ be a prime dividing both $|A/Z|$ and $|Z|$. Since $|A|$ is odd, the prime $r$ is odd. Choose $a\in A\setminus Z$ such that $aZ$ has order $r$. Then $a^r\in Z$ and $C_G(a)=A$.

    By Lemma~\ref{lemma:5-2}, we have $Z\le T$, so $A=T$. Hence $a\in T$ and $a$ is semisimple. The element $a$ must be regular, since otherwise Lemma~\ref{lemma:nr-ss--unip} would give a nonidentity unipotent element in $C_G(a)=A$, contradicting the oddness of $|A|$.

    Since the odd prime $r$ divides $|Z|$, inspection of the possible centers in \cite[Table~9.2]{MalleTesterman2011} shows that the Lie type is one of $A_\ell$, ${}^2A_\ell$, $E_6$, or ${}^2E_6$. In types $E_6$ and ${}^2E_6$, we have $r=3$. The image of $a$ in the adjoint group is then a nontrivial element of order $3$, and the argument in Lemma~\ref{lemma:E6-adjoint--impossible} shows that it is not regular, a contradiction.

    Suppose that the type is $A_\ell$ or ${}^2A_\ell$. By Lemma~\ref{lemma:aux}, we have $n=r$, the group $\sG$ is simply connected, and $T$ is diagonal. Since $n=r\ge5$, choose an element $\omega$ of order $r$ in $\F_q^\times$ in the linear case or in $\{c\in\F_{q^2}^\times\mid c^{q+1}=1\}$ in the unitary case. The diagonal element
    \[
        b=\operatorname{diag}(\omega,\omega,\omega^{-2},1,\ldots,1)
    \]
    lies in $T\setminus Z$ and has a repeated eigenvalue, so it is nonregular. By Lemma~\ref{lemma:nr-ss--unip}, the centralizer $C_G(b)$ contains a nonidentity unipotent element and is therefore strictly larger than the odd-order torus $T=A$. This contradicts the assumption that $A$ is a direct CC-torus.
\end{proof}

\begin{lemma}\label{lemma:6-1}
    Suppose that $A>Z$ is a direct CC-torus, where $|A/Z|$ is odd and coprime to $|Z|$. Then every element of $A^G\setminus Z$ has conjugacy class size $|G|/|A|$ and affords the same isolated vertex of $\cD(G)$.
\end{lemma}

\begin{proof}
    Let $a\in A^G\setminus Z$, so $a=b^g$ for some $b\in A\setminus Z$ and $g\in G$. Then $C_G(b)=A$, and therefore $C_G(a)=C_G(b)^g=A^g$. Hence $|C_G(a)|=|A|$ and $|a^G|=|G|/|A|$.

    It remains to show that this vertex is isolated. Let $x\in G\setminus Z$ satisfy $|C_G(x)|\mid |A|$ or $|A|\mid |C_G(x)|$, and set $\pi=\pi(|A/Z|)$.

    Suppose first that $|C_G(x)|\mid |A|$. Then $|C_G(x)/Z|\mid |A/Z|$. Therefore $xZ$ is a nonidentity $\pi$-element of $G/Z$, so by Lemma~\ref{lemma:williams} it is conjugate into $A/Z$. Hence $x\in A^G$, and the direct CC property gives $|C_G(x)|=|A|$.

    Now suppose that $|A|\mid |C_G(x)|$. Then $|A/Z|\mid |C_G(x)/Z|$. If $xZ$ is a $\pi$-element, then it is conjugate into $A/Z$, and again $|C_G(x)|=|A|$. Otherwise, choose a prime $b\mid |xZ|$ with $b\notin\pi$, and choose $r\in\pi$. Since $r\mid |C_G(x)/Z|$, there exists $yZ\in C_G(x)/Z$ of order $r$. Let $uZ$ be a power of $xZ$ of order $b$. Then $uZ$ and $yZ$ commute and have coprime orders, so their product has order $br$. This contradicts the fact that $\pi$ is a complete connected component of the prime graph of $G/Z$.

    Thus every vertex adjacent to $|G|/|A|$ is equal to it, and the vertex is isolated.
\end{proof}

\begin{lemma}\label{lemma:non-direct-cc}
    Suppose that $|C_G(s)|$ is odd and $C_G(s)=A$. If $A$ is not a direct CC-torus, then either $|s^G|\in\cU$, or there exist an odd prime $r$ and $\varepsilon\in\{1,-1\}$ such that $G=\PGL_r(q)$ when $\varepsilon=1$ or $G=\PGU_r(q)$ when $\varepsilon=-1$, and, with $N=(q^r-\varepsilon^r)/(q-\varepsilon)$, the vertex $|s^G|$ lies in the two-vertex component
    \[
        \left\{\frac{|G|}{rN},\frac{|G|}{N}\right\}.
    \]
\end{lemma}

\begin{proof}
    By assumption, there exists $a\in A\setminus Z$ such that $C_G(a)\ne A$. Since $A\le C_G(a)$, we have $|C_G(s)|=|A|\mid |C_G(a)|$, and hence $s\sim a$. Moreover, the inequality $C_G(a)>A$ shows that $|a^G|\ne |s^G|$.

    By Lemma~\ref{lemma:5-2}, we have $Z\le T$, so $A=T$. Thus $a\in T$ and is semisimple. If $a$ is nonregular, then Proposition~\ref{prop:noncent-nonreg-ss--in-U} gives $|a^G|\in\cU$, and hence $|s^G|\in\cU$. If $a$ is regular and $|C_G(a)|$ is even, then Proposition~\ref{prop:noncent-ss-|C|-even--in-U} gives the same conclusion.

    Suppose that $a$ is regular and $|C_G(a)|$ is odd. Proposition~\ref{prop:sec-5} applies. Its isolated alternative is impossible because the distinct vertices $|s^G|$ and $|a^G|$ are adjacent. Therefore either $|a^G|\in\cU$, in which case $|s^G|\in\cU$, or $|a^G|$ lies in the two-vertex component displayed in the statement. Since $s\sim a$, the same is true of $|s^G|$.
\end{proof}

The following proposition follows formally from Proposition~\ref{prop:sec-5} and the lemmas of this section.

\begin{prop}\label{prop:sec-6}
    Let $s\in G\setminus Z(G)$ be regular and semisimple with $|C_G(s)|$ odd. Then one of the following holds.
    \begin{itemize}
        \item[(1)] The vertex $|s^G|$ lies in $\cU$.
        \item[(2)] The vertex $|s^G|$ is isolated.
        \item[(3)] There exist an odd prime $r$ and $\varepsilon\in\{1,-1\}$ such that $G=\PGL_r(q)$ when $\varepsilon=1$ or $G=\PGU_r(q)$ when $\varepsilon=-1$, and, with $N=(q^r-\varepsilon^r)/(q-\varepsilon)$, the vertex $|s^G|$ lies in the two-vertex component
        \[
            \left\{\frac{|G|}{rN},\frac{|G|}{N}\right\}.
        \]
    \end{itemize}
\end{prop}

\section{Proof of the Main Theorems}\label{sec:7}

\begin{thm}\label{thm:components}
    Let $G=\sG^F$ be one of the finite groups of Lie type under consideration, in defining characteristic $2$. Then $\cD(G)$ has a distinguished component $\cU$ containing all nonidentity unipotent class sizes. Every component other than $\cU$ is either an isolated vertex or a two-vertex component of the form
    \[
        \left\{\frac{|G|}{rN},\frac{|G|}{N}\right\},
    \]
    where $G=\PGL_r(q)$ or $G=\PGU_r(q)$, the integer $r$ is an odd prime dividing $q-\varepsilon$, $\varepsilon=1$ in the linear case and $\varepsilon=-1$ in the unitary case, and $N=(q^r-\varepsilon^r)/(q-\varepsilon)$.
\end{thm}

\begin{proof}
    Let $x\in G\setminus Z(G)$, and write its Jordan decomposition as $x=su=us$. The element $s$ has odd order and $u$ has $2$-power order. Suppose first that $u\ne1$. If $s\in Z(G)$, then $C_G(x)=C_G(u)$, and hence $x\sim u$. If $s\notin Z(G)$, then Lemma~\ref{lemma:4-AIN} gives $s\sim x\sim u$. In either case, Proposition~\ref{prop:unip-component} gives $|x^G|\in\cU$.

    Now suppose that $u=1$, so $x$ is semisimple. If $x$ is nonregular, then Proposition~\ref{prop:noncent-nonreg-ss--in-U} gives $|x^G|\in\cU$. If $x$ is regular and $|C_G(x)|$ is even, then Proposition~\ref{prop:noncent-ss-|C|-even--in-U} gives $|x^G|\in\cU$. Finally, if $x$ is regular and $|C_G(x)|$ is odd, the conclusion follows from Proposition~\ref{prop:sec-6}.
\end{proof}

\begin{thm}\label{thm:isolated-vertices}
    The isolated vertices of $\cD(G)$ outside $\cU$ are exactly the following.
    \begin{itemize}
        \item[(1)] The vertex $|G|/|A|$, where $A=TZ(G)>Z(G)$, $T=\sT^F$ for an $F$-stable maximal torus $\sT\le\sG$, the subgroup $A$ is a direct CC-torus, and $|A|$ is odd.
        \item[(2)] The vertex $|G|/r^r$, where either $G=\PGL_r(q)$ and $r=q-1$ is an odd prime, or $G=\PGU_r(q)$ and $r=q+1$ is an odd prime.
    \end{itemize}
\end{thm}

\begin{proof}
    Suppose first that $A$ is as in~(1). By Lemma~\ref{lemma:corpimality}, we have $(|A/Z(G)|,|Z(G)|)=1$. Lemma~\ref{lemma:6-1} therefore shows that every element of $A^G\setminus Z(G)$ has class size $|G|/|A|$ and that this vertex is isolated. Since $|A|$ is odd whereas every nonidentity unipotent centralizer has even order, this vertex does not lie in $\cU$.

    Now suppose that $G$ is as in~(2). Let $\varepsilon=1$ in the linear case and $\varepsilon=-1$ in the unitary case, and let $K=\F_q^\times$ in the linear case and $K=\{c\in\F_{q^2}^\times\mid c^{q+1}=1\}$ in the unitary case. Then $|K|=r$. Choose a generator $\omega$ of $K$, and let $t$ be the image in $G$ of
    \[
        \operatorname{diag}(1,\omega,\omega^2,\ldots,\omega^{r-1}).
    \]
    The element $t$ is rational, regular, and semisimple, and $F$ fixes its eigenlines. Set $\sT=C_\sG(t)^\circ$ and $T=\sT^F$. The cyclic permutation of the eigenlines centralizes $t$ projectively, so $C_G(t)>T$. The calculation in Lemma~\ref{lemma:5-8} therefore applies and shows that $|C_G(t)|=r^r$ and that $|t^G|=|G|/r^r$ is isolated.

    Conversely, let $|x^G|$ be an isolated vertex outside $\cU$, and write the Jordan decomposition of $x$ as $x=su=us$. If $u\ne1$, then the proof of Theorem~\ref{thm:components} gives $|x^G|\in\cU$, a contradiction. Hence $x$ is semisimple. Proposition~\ref{prop:noncent-nonreg-ss--in-U} shows that $x$ is regular, and Proposition~\ref{prop:noncent-ss-|C|-even--in-U} shows that $|C_G(x)|$ is odd.

    Let $\sT=C_\sG(x)^\circ$, set $T=\sT^F$, and set $A=TZ(G)$. If $C_G(x)>A$, then Proposition~\ref{prop:sec-5} applies. Its first alternative is impossible because $|x^G|\notin\cU$, and its third alternative is impossible because $|x^G|$ is isolated. Therefore its second alternative holds, giving a vertex of type~(2).

    Suppose instead that $C_G(x)=A$. If $A$ is not a direct CC-torus, then Lemma~\ref{lemma:non-direct-cc} places $|x^G|$ either in $\cU$ or in one of the two-vertex components, both of which are impossible. Hence $A$ is a direct CC-torus. Since $x\in A\setminus Z(G)$, we have $A>Z(G)$, and since $A=C_G(x)$, the group $A$ has odd order. Thus $|x^G|$ is of type~(1).
\end{proof}

\newpage
\begin{appendices}

\section{Frobenius Groups} \label{app:frobenius}

Another apparent gap in the existing literature surrounding the divisibility graph $\cD(G)$ is the treatment of Frobenius groups. In this appendix, we provide a short proof of the fact that the divisibility graph of a Frobenius group has two components.

\begin{thm}\label{thm:frob-grp-divis-graph}
    Let $G = K\rtimes H$ be a Frobenius group with kernel $K$ and complement $H$. Then
    \[
        c(D(\cs(G))) = 2.
    \]
\end{thm}

\begin{proof}
    Set $A_K = \{[K:C_K(k)] : k\in K^\#\}$ and $A_H = \{[H:C_H(h)] : h\in H^\#\}$. Let $\cD^*(A_H)$ denote the divisibility graph on $A_H$, without excluding the vertex $1$ if $1\in A_H$.

    Recall the following facts about Frobenius groups: for all $h\in H^\#$, $C_K(h)=1$, equivalently $C_H(k)=1$ for all $k\in K^\#$; every element of $G\setminus K$ is conjugate to an element of $H^\#$; the kernel $K$ is nilpotent; $\gcd(|K|,|H|)=1$; and the complement $H$ has nontrivial center.

    First we compute the class sizes. Let $k\in K^\#$. We claim that $C_G(k)=C_K(k)$. Clearly $C_K(k)\le C_G(k)$. Conversely, if $x\in C_G(k)$ and $x\notin K$, then $x$ is conjugate to some $h\in H^\#$. Hence $h$ centralizes the nonidentity element $k^g\in K^\#$ for some $g\in G$, contradicting $C_H(k^g)=1$. Thus $x\in K$, and so $C_G(k)=C_K(k)$.

    Similarly, for $h\in H^\#$, we claim that $C_G(h)=C_H(h)$. The inclusion $C_H(h)\le C_G(h)$ is clear. If $x\in C_G(h)$, write $x=ab$ with $a\in K$ and $b\in H$. Passing to $G/K\cong H$, we get $b\in C_H(h)$. Therefore $a=xb^{-1}\in C_G(h)\cap K=C_K(h)=1$, so $x=b\in C_H(h)$.

    Therefore
    \[
        |k^G|=|G:C_G(k)|=|G:C_K(k)|=|H||K:C_K(k)|
    \]
    for $k\in K^\#$, and
    \[
        |h^G|=|G:C_G(h)|=|G:C_H(h)|=|K||H:C_H(h)|
    \]
    for $h\in H^\#$. Hence
    \[
        \cs(G)\setminus\{1\}=|H|A_K\cup |K|A_H.
    \]

    This union is disjoint, and there are no divisibility edges between the two parts. Indeed, take $ a=[K:C_K(k)]\in A_K $ and $b=[H:C_H(h)]\in A_H$. Then $a\mid |K|$, $b\mid |H|$, and $\gcd(|K|,|H|)=1$. Also $a<|K|$ and $b<|H|$, since $k\in C_K(k)$ and $h\in C_H(h)$ are nonidentity. If $|H|a\mid |K|b$, then $|H|\mid b$, a contradiction. Similarly, if $|K|b\mid |H|a$, then $|K|\mid a$, a contradiction.

    Now the induced subgraph on $|H|A_K$ is connected. Since $K$ is nilpotent, $Z(K)\ne 1$. Choose $1\ne z\in Z(K)$. Then $C_K(z)=K$, so $1=[K:C_K(z)]\in A_K$. Hence $|H|\in |H|A_K$, and for every $a\in A_K$ we have
    \[
        |H|\mid |H|a.
    \]
    Thus every vertex in $|H|A_K$ is adjacent to $|H|$.

    The induced subgraph on $|K|A_H$ is naturally isomorphic to $\cD^*(A_H)$, since $a\mid b$ if and only if  $\left(|K|a\right)\mid \left(|K|b\right)$. Finally, since the Frobenius complement $H$ has nontrivial center, choose $1\ne h\in Z(H)$. Then $C_H(h)=H$, so $1\in A_H$. Thus $1$ is adjacent to every other vertex of $\cD^*(A_H)$, and $\cD^*(A_H)$ is connected.

    Therefore $D(\cs(G))$ has exactly two connected components: one on $|H|A_K$ and one on $|K|A_H$.
\end{proof}

It is likely that similar results may hold of $2$-Frobenius groups or quasi-Frobenius groups. As this subject was not the primary focus of this paper, we leave these as questions to the reader.

\begin{question}
    Let $G$ be a $2$-Frobenius group. What is the number of components of $\cD(G)$?
\end{question}

\begin{question}
    Let $G$ be a quasi-Frobenius group, that is $G/Z(G)$ is a Frobenius group. What is the number of components of $\cD(G)$?
\end{question}

\end{appendices}

\section*{Acknowledgments}
The work was supported by the NSF under grant DMS-2447229 as well as Jane Street. The authors gratefully acknowledge the financial support of NSF and Jane Street, and thank Texas State University for providing a great working environment and support.
Yong Yang was also partially supported by a grant from the Simons Foundation (\#918096, to YY). Liam May was also partially supported by the MIT Department of Mathematics.

\section*{Disclosure Statement}
The authors declare that they have no competing interests and no conflicts of interest.

\section*{Data Availability Statement} Data sharing is not applicable to this article, as no data sets were generated or analysed during the current study.

\newpage
\bibliographystyle{amsalpha}
\bibliography{Bibliography}

\end{document}